\documentclass[11pt]{amsart}

\usepackage{graphicx}
\usepackage{tikz-cd}
\usepackage{amsfonts}
\usepackage{amsmath}
\usepackage{amssymb}
\usepackage{mathtools}
\usepackage{amsthm}
\usepackage[utf8]{inputenc}
\usepackage[english]{babel}
\newtheorem{theorem}{Theorem}[section]

\newtheorem{lemma}[theorem]{Lemma} 
\newtheorem{proposition}[theorem]{Proposition}
\newtheorem{definition}[theorem]{Definition}
\newtheorem{remark}[theorem]{Remark}

\newcommand{\Z}{\mathbb{Z}}

\newcommand{\Or}{\mathcal{O}}

\usepackage[
top    = 1.5in,
bottom = 1.5in,
left   = 1.5in,
right  = 1.5in]{geometry}

\title{Kasteleyn cokernels and perfect matchings on planar bipartite graphs}
\author{Libby Taylor}

\begin{document}

\begin{abstract}

The determinant method of Kasteleyn gives a method of computing the number of perfect matchings of a planar bipartite graph.  In addition, results of Bernardi exhibit a bijection between spanning trees of a planar bipartite graph and elements of its Jacobian.  In this paper, we explore an analogue of Bernardi's results, providing a canonical simply transitive group action of the Kasteleyn cokernel of a planar bipartite graph on its set of perfect matchings, when the planar bipartite graph in question is of the form $G^+$, as defined by Kenyon, Propp and Wilson.  

\end{abstract}

\maketitle

\section{Introduction}

In general, counting matchings of a bipartite graph is a $\#P$-complete problem; this was proved by Valiant in ~\cite{valiant}.  In the case that the graph is planar, however, Kasteleyn's theorem gives a method of enumerating matchings in polynomial time.  The enumeration involves calculating the determinant of a certain signed adjacency matrix, called the \emph{Kasteleyn matrix} of the graph; equivalently, the number of perfect matchings is equal to the order of the \emph{Kasteleyn cokernel}, which is defined as the cokernel of the Kasteleyn matrix.  Kuperberg discusses in ~\cite{kuperberg} the possibility of a natural bijection between the matchings of the graph and the elements of its Kasteleyn cokernel.  He suggests that it may be better to consider a quantum bijection between these two sets, that is, a unitary isomorphism between their formal complex linear spans.  In the case that the graph is of the form $G^+$, as defined in ~\cite{KPW}, we prove that the situation is simpler: there is a canonical simply transitive group action of the Kasteleyn cokernel on the set of perfect matchings.  Our description is analogous to certain families of combinatorial bijections between the spanning trees of a planar graph and elements of its Jacobian.  


\section{Background}

All graphs in this paper will be assumed to be finite and connected, possibly with multi-edges but without self-loops.  

\smallskip

Let $G$ be a directed graph on $n$ vertices.  Its $n \times n$ signed adjacency matrix $A$ is defined to to have its $(i,j)$-th entry in $A$ equal to $1$ if there is a directed edge from $v_i$ to $v_j$; $-1$ if there is a directed edge from $v_j$ to $v_i$; and $0$ if no edge exists between the two vertices.  If there are multiple edges between $v_i$ and $v_j$, then the matrix entry is equal to the number of edges oriented $v_i$ to $v_j$ minus the number of edges oriented $v_j$ to $v_i$.

\smallskip

A \emph{divisor} on a graph $G=(V,E)$ is a function $D: V\to \Z$.  The set of all divisors of a graph $G$ is denoted $\text{Div}(G)$.  Any $d\in \text{Div}(G)$ can be written as $d=\Sigma_{v\in V} a_v(v)$, for $a_v\in \Z$.  The \emph{\text{deg}ree} of the \text{div}isor $d$ is defined as $\text{deg}(d)=\Sigma_{v\in V} a_v$.  The set of all \text{div}isors of \text{deg}ree $k$ is denoted $\text{Div}^k(G)$.  There is a (non-canonical) map from $\text{Div}^k(G)\to \text{Div}^0(G)$ defined by $d\mapsto d-d_0$ for some fixed reference \text{div}isor $d_0\in \text{Div}^k(G)$.  

\smallskip

The \emph{Laplacian operator} on a graph $G$ is denoted $\Delta :Z(G)\to \text{Div}(G)$, where $Z(G)$ is the set of integer-valued functions on the set $V(G)$.  Whenever $f\in Z(G)$, the Laplacian operator is defined by $\Delta (f)=\Sigma_{v\in V(G)}\Delta_v (f)(v)$, where
\[
\Delta_v(f)=\Sigma_{(v,w)\in E(G)}(f(v)-f(w)).
\]
The group \text{Prin}$(G)$ of \text{prin}cipal \text{div}isors is the image of the Laplacian operator.  It is obvious that \text{Prin}$(G)\subseteq \text{Div}^0(G)$.  Both \text{Prin}$(G)$ and \text{Div}$^0(G)$ are free abelian groups of rank $n-1$,  so
\[
\text{Jac(}G)=\text{Div}^0(G)/\text{Prin}(G)
\]
is a finite group called the \textit{Jacobian} of $G$.

\subsection{Generalized Temperley Bijection}

In ~\cite{KPW}, Kenyon, Propp, and Wilson create a method for obtaining a planar bipartite graph from an arbitrary planar graph and exhibit a bijection, called the Temperley bijection, between the spanning trees of the original graph and the matchings of the new bipartite graph.  This section is a summary of this method.  Throughout, we denote the starting planar graph as $G$ and the resulting bipartite graph as $G^+$.  

\smallskip


Fix an embedding of $G$ in the plane (this process does depend on the chosen embedding).  Choose a vertex of $G$ which is adjacent to the infinite face of $G$ with respect to this embedding and call it $q$.

\smallskip

Overlay $G$ with its planar dual $G^\vee$ in the plane and denote as $q^*$ the vertex of $G^\vee$ corresponding to the infinite face of $G$.  At each intersection between an edge of $G$ and an edge of $G^\vee$, add a vertex in order to create a bipartite graph.  To complete the construction, delete $q$, $q^*$, and all the edges incident to either $q$ or $q^*$.  The resulting graph is called $G^+$.  The vertices of $G^+$ are partitioned into white and black.  The white vertices are those corresponding to edges of $G$ and the black vertices are those corresponding to vertices of either $G$ or $G^\vee$.  Each edge in $G^+$ is a half-edge in either $G$ or $G^\vee$.  See Figures ~\ref{G}, ~\ref{GUGV}, and ~\ref{G+} for an example of this method.
\begin{figure}[!htb] \includegraphics[scale=.5]{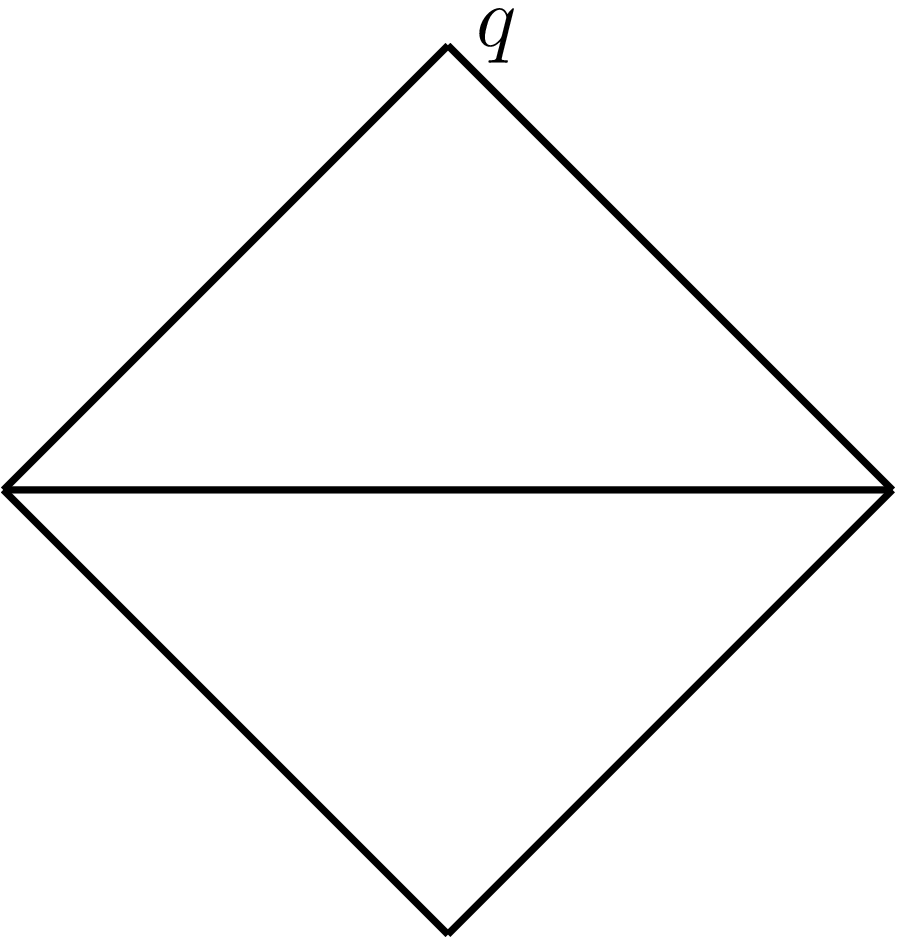} \caption{$G$} \label{G} \end{figure}
\begin{figure}[!htb] \includegraphics[scale=.5]{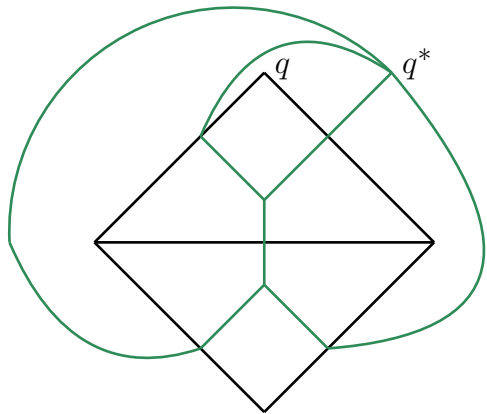} \caption{$G \cup G^\vee$}\label{GUGV} \end{figure}
\begin{figure}[!htb] \includegraphics[scale=.5]{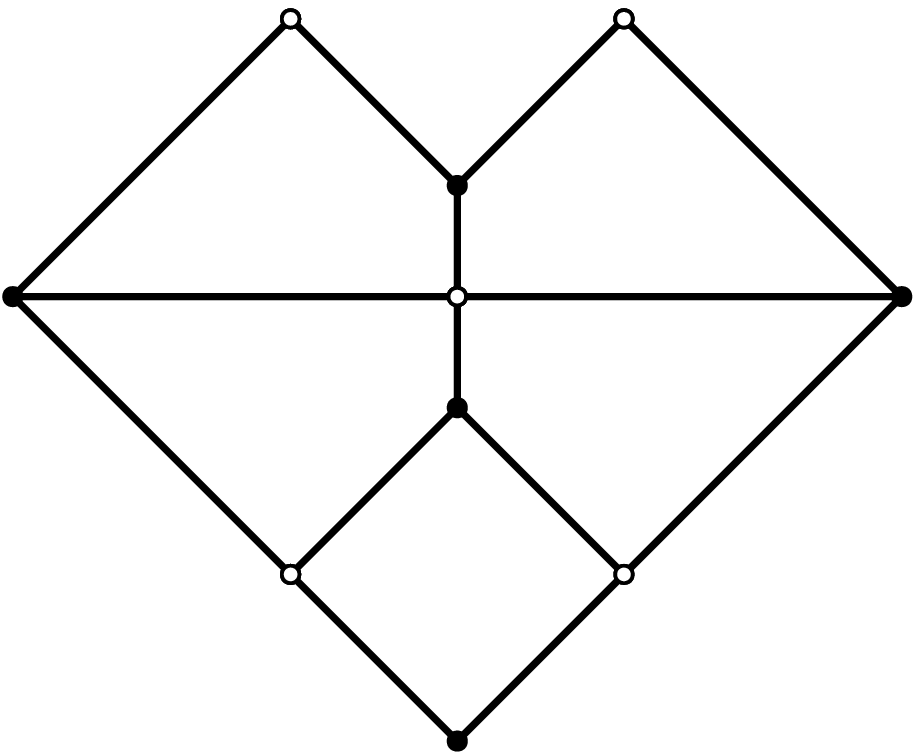} \caption{$G^+$} \label{G+} \end{figure}

Kenyon, Propp and Wilson in ~\cite{KPW} produce a bijection between the set of spanning trees of $G$ and the set of perfect matchings of $G^+$, denoted $T(G)$ and $M(G^+)$ respectively.  The starting information for this bijection is a spanning tree $T$ of $G$ and the root vertex $q$.  One then constructs a $q$-connected orientation of $G$ associated to $T$.  (By a $q$-connected orientation, we mean that for each $v\in V(G)$, there exists a directed path from $q$ to $v$.)  This orientation is constructed by first orienting all edges $e\in T$ away from $q$.  Then each $e\in T^c$ is oriented counterclockwise with respect to its fundamental cycle in $T$.  For any spanning tree $T_i\in T(G)$, we denote the associated $q$-connected orientation of $G$ by $O_i$.  It will turn out that every $q$-connected orientation of $G$ arises in this way from some spanning tree (see ~\cite{bernardi}), so the set of $q$-connected orientations of $G$ is in bijection with the set of spanning trees of $G$.  

\smallskip

Throughout, edges in $G^+$ which are oriented white $\to$ black will be referred to as positively oriented and edges oriented black $\to$ white as negatively oriented.

\smallskip

Now we can construct a matching in $M(G^+)$ from an element of $T(G)$; this construction will produce the desired bijection.  Start with some spanning tree of $G$, which we will call $T_0$, and the $q$-connected orientation $O_0$ constructed from $T_0$.  Orient the edges of $G^\vee$ counterclockwise from the orientation (in $O_0$) of the edge they intersect in $G$.  For each white vertex $v_w$ corresponding to some edge in $T_0$, choose the positively oriented half-edge of $G$ incident to $v_w$ and add it to $M(G^+)$.  For each $v_w$ corresponding to some edge in the complement of $T_0$, choose the positively oriented half-edge of $G^\vee$ incident to $v_w$ and add it to $M(G^+)$.  See Figures ~\ref{T_0} and ~\ref{M_0} for an example.  

\begin{remark}

The orientation of $G^+$ constructed as described above will turn out to be a Kasteleyn orientation; these orientations will be defined in Section 2.3.

\end{remark}

\begin{figure}[!htb] \includegraphics[scale=.5]{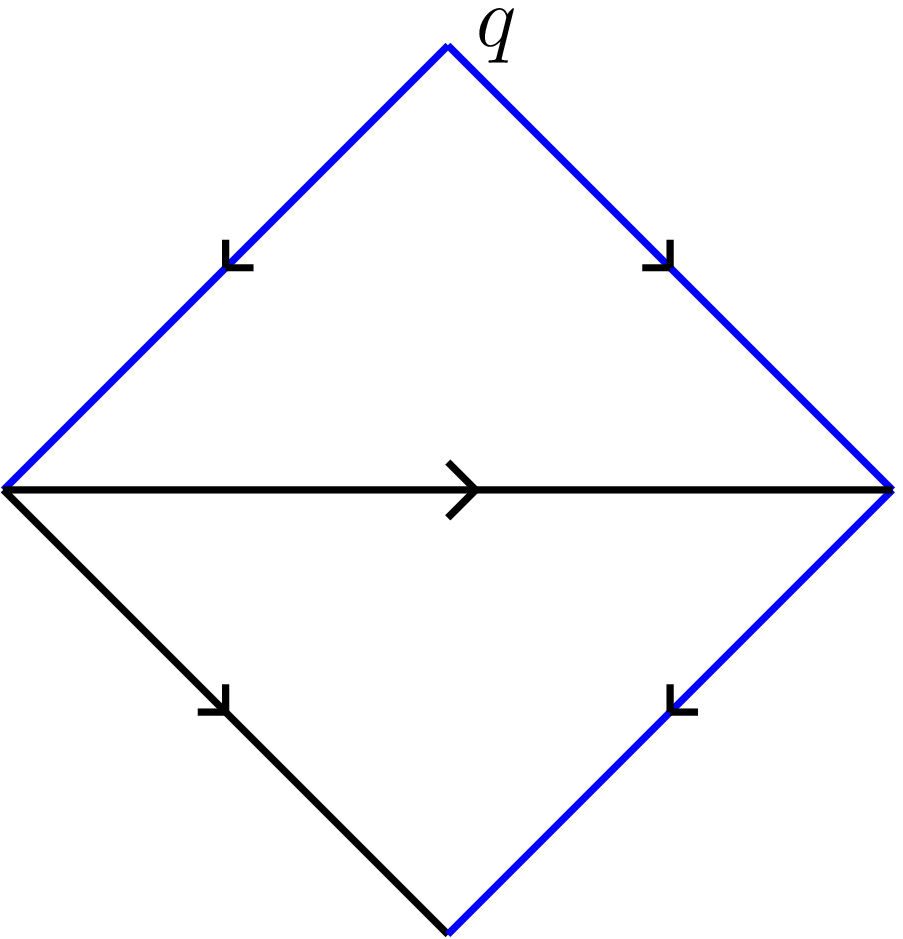} \caption{$T_0$} \label{T_0} \end{figure}
\begin{figure}[!htb]\includegraphics[scale=.5]{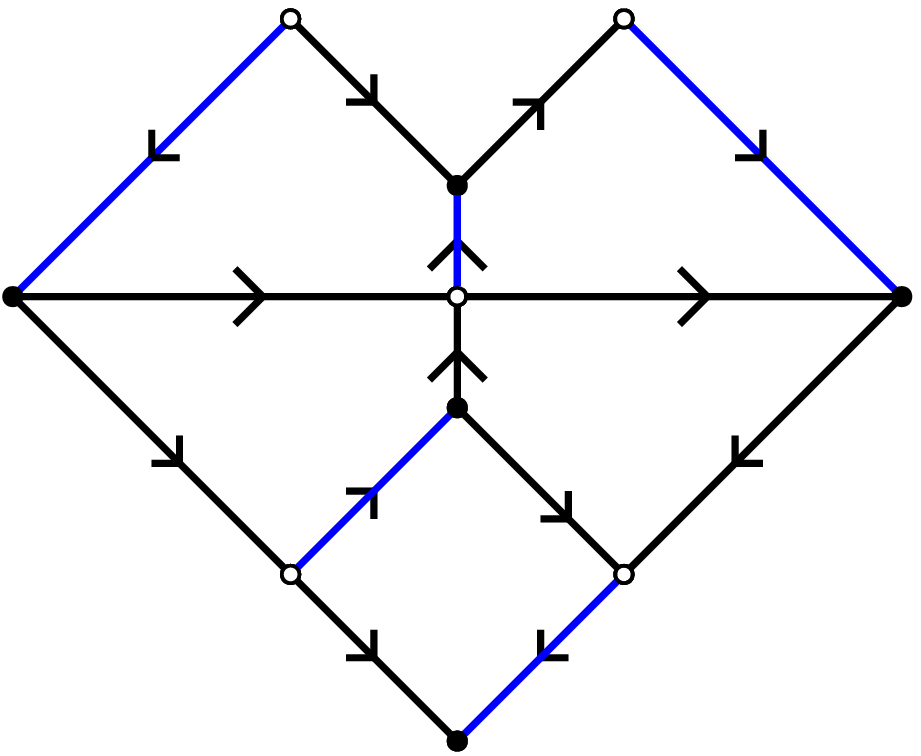} \caption{$M_0$} \label{M_0} \end{figure}

This construction gives rise to a bijection between spanning trees of $G$ and perfect matchings of $G^+$, provided that $q$ is incident to $q^*$ (by which we mean that $q$ is adjacent to the infinite face).  If $q$ is not incident to $q^*$, then the set $M(G^+)$ may be strictly larger than the set $T(G)$, in which case this map produces an injection $T(G)\to M(G^+)$.  Throughout, we assume that $q$ and $q^*$ are incident.

\subsection{Orientations of Spanning Trees}

In this section, we will describe the Bernardi bijection between $T(G)$ and elements of $\text{Jac(}G)$ and show that it factors through the set of equivalence classes of orientations of $G$.  
The bijection begins with some spanning tree $T_i$ of $G$.  Then a $q$-connected orientation is constructed as described in Section 2.1.  An equivalence relation can be defined on the orientations of $G$ by considering two orientations $O_i$ and $O_j$ to be equivalent if one can be obtained from the other by a sequence of directed cycle and cut reversals.  It is clear that this is an equivalence relation.  We will denote the set of equivalence classes of orientations of $G$ by $\Or(G)$.  The following proposition is proved in ~\cite{spencer}:

\begin{proposition}\label{q-conn orientations}

Each equivalence class in $\mathcal{O}(G)$ contains a unique $q$-connected orientation.

\end{proposition}

Bernardi provides a bijective proof of this fact in ~\cite{bernardi} by proving that $T(G)\longleftrightarrow \Or(G)$.  Since spanning trees of $G$ are in bijection with $q$-connected orientations, the $q$-connected orientations make a natural choice of representatives for the equivalence classes in $\Or(G)$.  

\smallskip

The bijection between $\Or(G)$ and $\text{Jac(}G)$ maps an orientation in $\Or(G)$ into $\text{Div}^{g-1}(G)$ by placing a coefficient of $\text{in\text{deg}}(v)-1$ on each vertex $v$, where $g$ denotes the combinatorial genus of $G$ and $\text{in\text{deg}}(v)$ denotes the in\text{deg}ree of $v$.  The resulting divisor has \text{deg}ree $\Sigma_{v\in V} (\text{in\text{deg}}(v)-1)=g-1$, so this does give a map $\Or(G) \to \text{Div}^{g-1}(G)$.   

\smallskip

The map from $\text{Div}^{g-1}(G)$ to $\text{Jac(}G)$ is defined by subtracting a reference divisor $d_0\in \text{Div}^{g-1}(G)$ from each $d\in \text{Div}^{g-1}(G)$.  The reference divisor is taken to be the divisor associated to some spanning tree $T_0$ of $G$, which allows $T_0$ to be considered as an ``identity element'' of $T(G)$ in this bijection.  This produces a bijection between $\text{Div}^{g-1}(G)/\text{Prin}(G)$ and $\text{Div}^0(G)/\text{Prin}(G) \cong \text{Jac(}G)$.  Thus the Bernardi map gives a combinatorially defined bijection between spanning trees of $G$ and elements of $\text{Jac}(G)$, which 
factors as $T(G) \to \Or(G) \to \text{Div}^{g-1}(G)/\text{Prin(G)} \to \text{Jac(}G)$.  Note that this bijection is not canonical, as it depends on the choice of $T_0$ and the choice of $q$.  However, Bernardi also proves that there is an associated group action of the so-called \textit{break divisors} of $G$ on the set of spanning trees, and the break divisors are in bijection with $q$-connected orientations.  Yuen proves in ~\cite{chiho} that this group action is independent of the reference tree, and depends only on the choice of $q$.  

\smallskip

By abuse of terminology, we will also refer to the truncated map $\Or(G) \to \text{Div}^{g-1}(G)/\text{Prin}(G) \to \text{Jac(}G)$ from equivalence classes of orientations to the Jacobian of $G$ as the Bernardi bijection.  

\subsection{Kasteleyn Cokernels and Jacobians}

\medskip

The Kasteleyn cokernel is closely related to the Kasteleyn orientations of a planar bipartite graph.  These objects arise as analogues of $\text{Jac(}G)$ and $q$-connected orientations of $G$, respectively.

\begin{definition}

A Kasteleyn orientation on a planar bipartite graph $G$ is an orientation of $G$ such that every cycle in the graph has an odd number of clockwise-oriented edges.

\end{definition}

(Note that this definition is dependent on the drawing of $G$ in the plane.)  The condition for an orientation to be Kasteleyn is equivalent to having an odd number of positively oriented edges in every cycle with length $\ell\equiv 0\pmod 4$ and an even number of positively oriented edges in every cycle with length $\ell\equiv 2\pmod 4$.  

\smallskip

A $q$-connected orientation on $G$ gives rise to a Kasteleyn orientation on $G^+$ using the same method as in the Temperley bijection.  First recall that all edges of $G^+$ are either half-edges of $G$ or half-edges of $G^\vee$.   Orient each half-edge of $G$ the same way as in the original $q$-connected orientation of $G$ and orient each half-edge of $G^\vee$ counterclockwise from its corresponding edge in $G$.  (This orientation was used in the Temperley bijection to produce a perfect matching on $G^+$ from a spanning tree on $G$.)  Figure ~\ref{M_0} shows the matching and orientation coming from $T_0$, and one can verify that the orientation induced on $G^+$ is a Kasteleyn orientation.  

\smallskip

The signed bipartite adjacency matrix of a graph is constructed with white vertices indexing the columns and black vertices indexing the rows.  A signed bipartite adjacency matrix arising from a Kasteleyn signing is called a Kasteleyn matrix.  The Kasteleyn cokernel is defined as follows.

\begin{definition}

The Kasteleyn cokernel of a planar bipartite graph $H$ is the finite abelian group $K(H)=\text{Div}(H)/\text{Prin}(H)$, where $\text{Div}(H)$ is the free abelian group on the white vertices and $\text{Prin}(H)$ is the column span of a Kasteleyn matrix of $H$.  

\end{definition}

In general, Kasteleyn signings are far from being unique.  However, the Kasteleyn cokernel is independent of the Kasteleyn signing chosen; see ~\cite{kuperberg}.

\smallskip

Jacobson proves the following theorem in ~\cite{jacobson}, which was originally conjectured by Kuperberg:

\begin{theorem}\label{isomorphism}

The Kasteleyn cokernel of $G^+$ is isomorphic to the Jacobian of $G$, i.e. $K(G^+)\cong \text{Jac(}G)$.

\end{theorem}

In order to explicitly describe the isomorphism, first note that elements $a,b\in \text{Jac(}G)$ are equivalent if and only if one can be obtained from the other by a sequence of chip-firing moves.  A chip-firing move from a fixed vertex $v_0$ has the form
\[
\Delta v_0= - |N(v_0)|; \Delta v=1 \ \text{for}\  v\in N(v_0)
\]
where $N(v_0)$ denotes the neighborhood of $v_0$.  
It is easy to see that chip-firing equivalence is an equivalence relation on the divisors of $G$ and that the \text{deg}ree of the divisor will not change in a chip-firing move. 

\smallskip

The isomorphism $\phi: K(G^+)\to \text{Jac(}G)$ is defined as follows.  Let $\text{Div}(G^+)$ denote the set of linear equivalence classes of white vertices of $G^+$, where two divisors are said to be linearly equivalent if they differ by something in the column span of a Kasteleyn matrix of $G^+$.  

\smallskip

Let $k\in \text{Div}(G^+)$, and choose a representative for its equivalence class.  Denote the integer on a given white vertex $v_w$ as $d$.  Let $e$ denote the edge of $G$ associated with $v_w$, and place $d$ chips at the head of $e$ and $-d$ chips at the tail of $e$, where the heads and tails of each edge are determined by a fixed $q$-connected orientation on $G$.  Extend by linearity, and take the linear equivalence class of the resulting divisor on $G$ to be the image of $\phi(k)$.  The linear equivalence class of the resulting divisor is independent of the representative chosen for $k$, so $\phi$ is well-defined.  Then $\phi:\text{Div}(G^+) \to \text{Jac(}G)$ is a surjection whose kernel is exactly $\text{Prin}(G^+)$, which gives the desired isomorphism.


\section{Group action of the Kasteleyn cokernel on matchings}

In this section, we establish the following theorem.

\begin{theorem}

There is a canonical simply transitive group action of the Kasteleyn cokernel of $G^+$ on the set of perfect matchings of $G^+$.

\end{theorem}


We will show that the group action depends only on the choice of root vertex $q$, which we consider to be part of the structure of $G^+$.  To describe the action, start by fixing a reference matching $M_0$ on a planar bipartite graph $G^+$, and fix the Kasteleyn orientation coming from $M_0$ via the corresponding $T_0\in T(G)$.  (Later, we will show that the group action is independent of the choice of $M_0$.)  Define the alternating cycles of $G^+$ (with respect to $M_0$) to be cycles whose edges alternate between edges $e\in M_0$ and $e\not\in M_0$.  For any matching $M_j$ in $M(G^+)$, the symmetric difference $M_0 \triangle M_j$ is some disjoint union of alternating cycles.    More generally, for any 2 matchings $M_i$ and $M_j$, the symmetric difference $M_i \triangle M_j$ is a disjoint union of cycles of $G^+$.  We define $L_{ij}:=M_i \triangle M_j$. 

We define a map
\[
\Psi_0: M(G^+) \to K(G^+)
\]
as follows.  For each white vertex $v$ in the support of $L_{0i}$, if the 2 edges in $L_{0i}$ incident to $v$ have the same orientation (that is, both are positively oriented or both are negatively oriented), then a $0$ is placed on $v$.  If the 2 edges have opposite orientations (that is, one positive and the other negative), then a 1 is placed on $v$.  For any $v$ not in the support of $L_{0i}$, a 0 is placed on $v$.  This gives a divisor on $G^+$; take $\Psi_0(M_i)$ to be the linear equivalence class of this divisor, which gives an element of $K(G^+)$.  Note that $\Psi_0(M_0)=0$.  
See Figures ~\ref{T_0}, ~\ref{M_0}, ~\ref{T_1}, ~\ref{M_1} and ~\ref{L_1} for an example of two spanning trees and the disjoint union of alternating cycles arising from the corresponding matchings on $G^+$.

\smallskip

$\Psi_0$ naturally extends to a map $\psi_0: M(G^+) \times M(G^+) \to K(G^+)$ by  letting $\psi_0(M_i,M_j) = \Psi_0(M_i)-\Psi_0(M_j)$.  Both $\Psi_0(M_i)$ and $\Psi_0(M_j)$ are elements of the Kasteleyn group, and the subtraction is performed in $K(G^+)$.  It is clear from the definition that $\psi_0$ defines a group action of $K(G^+)$ on $M(G^+)$.  We will prove that the action is simply transitive by proving that $\Psi_0$ is a bijection.  We will later prove that this action is independent of the reference data, i.e. the choice of $M_0$ and the corresponding Kasteleyn orientation, so in fact this group action is canonical.

\begin{theorem}

The map $\Psi_0: M(G^+) \to K(G^+)$ is a bijection between $K(G^+)$ and $M(G^+)$.  

\end{theorem}


Given a matching $M_i\in M(G^+)$, denote as $T_i$ the corresponding spanning tree of $G$ under the Temperley bijection, and denote as $O_i$ the $q$-connected orientation of $G$ associated to $T_i$.  We will show that the map $\Psi_0$ makes the following diagram commute:

\medskip

\begin{tikzcd}[row sep = huge, column sep = huge]
T(G) \arrow[d,"Temperley"] \arrow[r, "Bernardi"] & \text{Jac(}G) \arrow[d, "\phi"]\\
M(G^+) \arrow[r, "\Psi_0"] & K(G^+)\\
\end{tikzcd}

\vspace{-.3in}

Recall that since $T(G)$ is in natural bijection with the set $\Or(G)$ of equivalence classes of orientations of $G$ and the Bernardi bijection factors through this set, it is equivalent to state that the map $\Psi_0$ makes the following diagram commute:

\medskip

\begin{tikzcd}[row sep = huge, column sep = huge]
\Or(G) \arrow[d,"Temperley"] \arrow[r, "Bernardi"] & \text{Jac(}G) \arrow[d, "\psi"]\\
M(G^+) \arrow[r, "\Psi_0"] & K(G^+)\\
\end{tikzcd}

\vspace{-.3in}

\begin{figure} [!htb] \includegraphics[scale=.5]{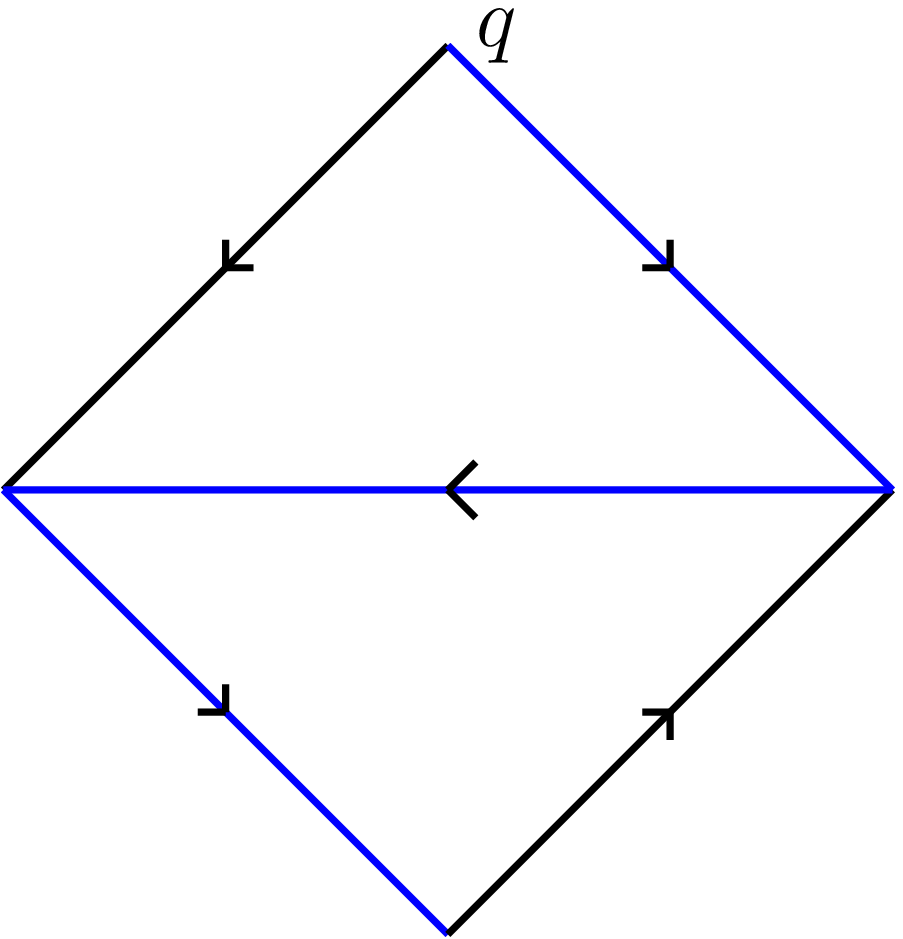} \caption{$T_1$} \label{T_1} \end{figure}
\begin{figure}[!htb] \includegraphics[scale=.5]{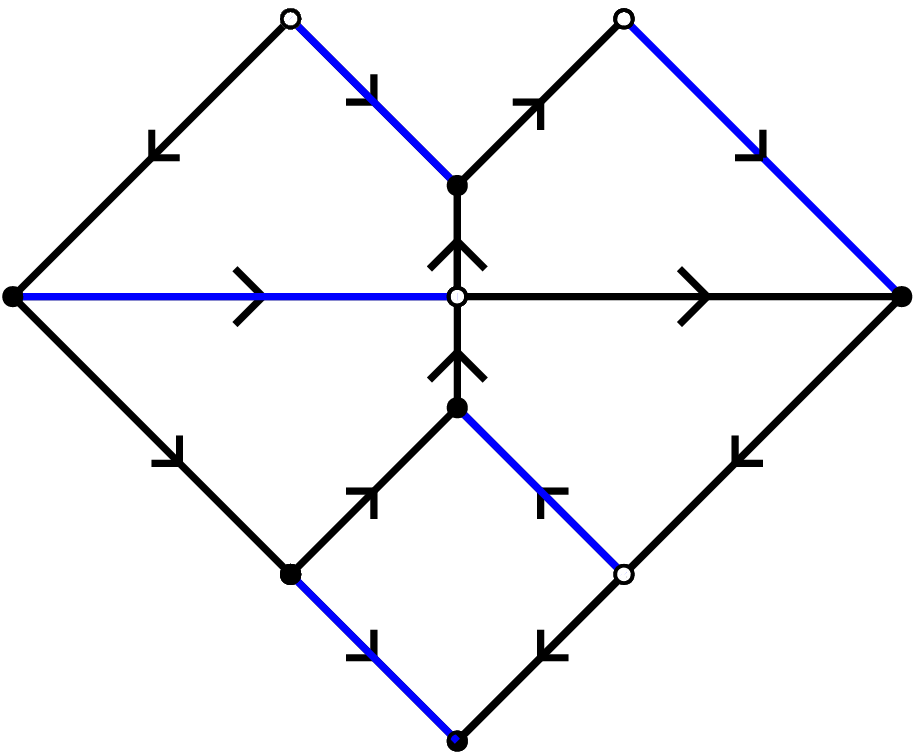} \caption{$M_1$} \label{M_1} \end{figure}
\begin{figure} [!htb] \includegraphics[scale=.5]{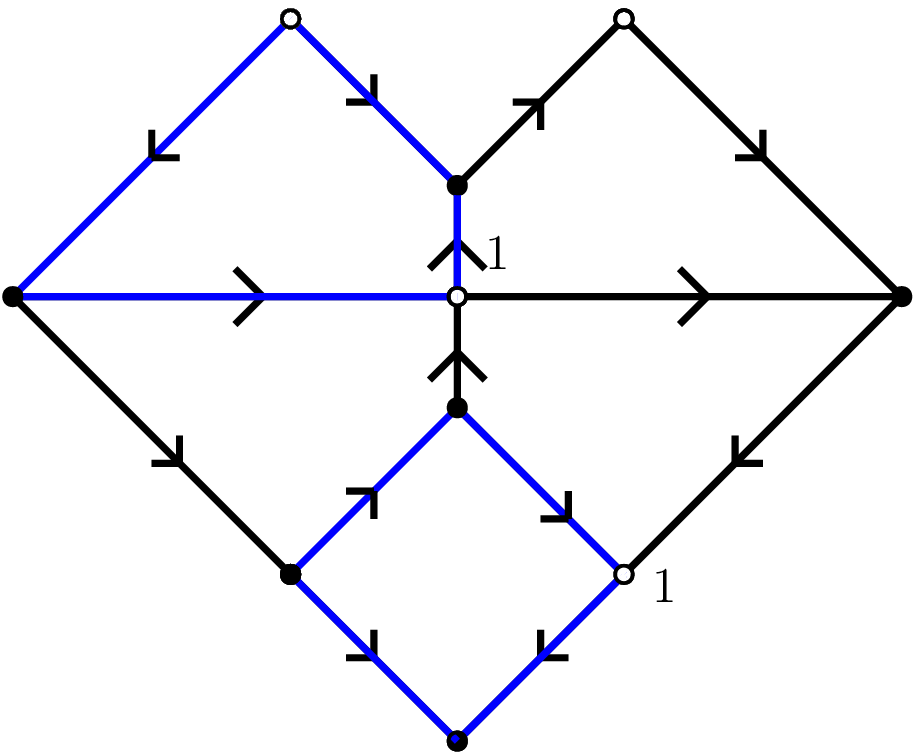} \caption{$L_{01}$} \label{L_1} \end{figure}

\begin{lemma}\label{edge reversal}

The map $\Psi_0$ is the same as the map induced by placing a 1 on each edge of $G$ which changes orientation between $O_0$ and $O_i$.  

\end{lemma}

\begin{proof} The statement holds by the isomorphism $\phi:\text{Jac(}G) \to K(G^+)$, since flipping the orientation of an edge $e \in G$ decreases the in\text{deg}ree of one endpoint of $e$ while increasing the in\text{deg}ree of the other endpoint of $e$.   Therefore $\Psi$ coincides with the map $\Or(G) \to \text{Jac(}G) \to K(G^+)$ given by composing the Bernardi bijection $\Or(G) \to \text{Jac}(G)$ with Jacobson's bijection $\phi:\text{Jac(}G) \to K(G^+)$.  
\end{proof}

For an example, see Figures ~\ref{T_0}, ~\ref{T_1} and ~\ref{L_1}.  Note that the edges which have different orientations in $O_0$ and $O_1$ are exactly those which have 1's placed on them when $\Psi_0$ is applied to $L_1$.

\begin{theorem}\label{bijection}

The map $\Psi_0$ makes the diagrams above commute, and therefore produces a bijection between $M(G^+)$ and $K(G^+)$.

\end{theorem}

\begin{proof}

Consider the Kasteleyn orientation of $G^+$ created from $O_0(G)$, as described in Section 2.1.  Under the Temperley bijection, each edge taken in a matching $M_i$ is 
positively oriented in the Kasteleyn orientation of $G^+$ arising from $O_i$.  Suppose a 1 is placed on a vertex $v$ in $L_i$.  Then the orientation of that edge is different in $O_0$ and $O_i$, since the Temperley bijection always picks up positively oriented edges.  Now suppose that there is a 0 on some vertex $v$ in $L_i$.  
Then the orientation of that edge is the same in $O_0$ and $O_i$.  Since it is in the symmetric difference $M_0 \Delta M_i$, the edge was either in $T_0$ but not in $T_i$, or vice versa, but its orientation remained the same in the two corresponding $q$-connected orientations since the same half-edges are still positively oriented.


Therefore, this bijection between $K(G^+)$ and $M(G^+)$ makes the diagram commute.  Since the other 3 arrows in the diagram are all bijections, $\Psi_0$ is as well.  

\end{proof}

See Figures ~\ref{T_0}, ~\ref{T_1}, and ~\ref{L_1} for an illustration of this statement.  The central edge and bottom right edge are the ones whose orientations are reversed between the orientations of $T_0$ and $T_1$, and the white vertices in $G^+$ corresponding to those edges are exactly the vertices which will have 1's placed on them under the map $\Psi_0$.  So in this example, the map $\Psi_0$ does in fact complete the commutative diagram.

\smallskip

Last, we show that the induced group action $\psi_0$ is independent of the reference matching $M_0$.  Suppose that some $M_i$ is used as the reference matching instead of $M_0$, so the group action $\psi_i$ is induced by the map $\Psi_i$ sending $M_j \to L_{ij} \to K(G^+)$, and the map $L_{ij} \to K(G^+)$ is defined with respect to the Kasteleyn orientation arising from $M_i$ via the Temperley bijection.  




The action described of $K(G^+)$ on $M(G^+)$ is equivalent to the action of break divisors on spanning trees of $G$, and this action is canonical (see ~\cite{chiho}), i.e., independent of the reference spanning tree.  Therefore the action $\psi_0$ is independent of $M_0$, and in fact depends only on the choice of $q$, which we consider to be part of the data of $G^+$ as a planar graph.  


Therefore $\psi_0=\psi_i$, so this defines a canonical group action, which we denote $\psi$, of $K(G^+)$ on $M(G^+)$.  

\smallskip

We note that this algorithm does not extend to graphs not of the form $G^+$.  It would be interesting to know whether a similar algorithm exists for more general graphs.

\section{acknowledments}

The author would like to thank Matt Baker for suggesting the problem and for many helpful conversations and Marcel Celaya for helpful feedback on an earlier version of the paper.  

\end{document}